\documentclass{article}
\usepackage{amsmath}
\usepackage[dvips]{graphicx}

\usepackage{color}
\usepackage{setspace}  
\usepackage{amscd}
\usepackage{amsthm}
\usepackage{enumerate}
\usepackage[all]{xy}
\usepackage{amssymb}
\usepackage{graphicx}

\input xy
\xyoption{all}
\title{Sharply 2-transitive linear groups}
\author{Yair Glasner and Dennis D. Gulko}
\date{\today}

\newtheorem{theorem}{Theorem}[section]
\newtheorem{lemma}[theorem]{Lemma}
\newtheorem{prop}[theorem]{Proposition}
\newtheorem{defi}[theorem]{Definition}
\newtheorem{cor}[theorem]{Corollary}
\newtheorem{conj}[theorem]{Conjecture} 

\newtheorem{ex}[theorem]{Example}

\newtheorem{rmk}[theorem]{Remark}

\newcommand{\Qed}{\nobreak \ifvmode \relax \else
      \ifdim\lastskip<1.5em \hskip-\lastskip
      \hskip1.5em plus0em minus0.5em \fi \nobreak
      \vrule height0.75em width0.5em depth0.25em\fi}
      
\newcommand{\N}{{\mathbb{N}}}

\newcommand{\R}{{\mathbb{R}}}

\newcommand{\GG}{{\mathbb{G}}}

\newcommand{\arrow}{\rightarrow}

\newcommand{\Sym}{{\operatorname{Sym}}}

\newcommand{\action}{\curvearrowright}
\newcommand{\GL}{{\operatorname{GL}}}

\newcommand{\Inv}{{\operatorname{Inv}}}
\newcommand{\Char}{{\operatorname{char}}}

\newcommand{\ord}{{\operatorname{Ord}}}
\newcommand{\gtd}{{\operatorname{gtd}}}

\newcommand{\pchar}{{{\it {p}}\operatorname{-char}}}

\def\ga{\alpha}

\def\bE{\begin{enumerate}}
\def\eE{\end{enumerate}}
\def\ol{\overline}

\newcommand{\LR}[1]{\left\{#1\right\}}

\newcommand{\bit}{\begin{itemize}}
\newcommand{\eit}{\end{itemize}}

\begin{document}
\bibliographystyle{number}
\maketitle

\section*{Abstract} A group $\Gamma$ is sharply $2$-transitive if it admits a faithful permutation representation that is transitive and free on pairs of distinct points. Conjecturally, for all such groups there exists a near-field $N$ (i.e. a skew field that is distributive only from the left, see Definition \ref{NF}) such that $\Gamma\cong N^\times\ltimes N$. This is well known in the finite case. We prove this conjecture when $\Gamma < \GL_n(F)$ is a linear group, where $F$ is any field with $\Char(F)\neq 2$ and that $\pchar(\Gamma) \neq 2$ (see Definition \ref{def_prim}).

\section{Introduction}
A sharply $k$-transitive group is, by definition, a permutation group which acts transitively and freely on ordered $k$-tuples of distinct points. Quite early on it was realized that $k$ is very limited; in his 1872 paper, \cite{J}, Jordan proved that finite sharply k-transitive groups with $k \geq 4$ are either symmetric, alternating or one of the Mathieu group $M_{11},M_{12}$. In the infinite case it was proved by J. Tits, in \cite{T1}, and M. Hall, in \cite{Ha}, that $k \leq 3$ for every infinite sharply $k$-transitive group. 

Sharply $2$-transitive groups attracted the attention of algebraists for many years because they lie on the borderline of permutation group theory and abstract algebraic structures. It is easy to see that if $K$ is a skew field then the semi-direct product $K^{\times} \ltimes K$ admits a sharply 2-transitive action on $K$, in which $K$ acts by addition and $K^{\times}$ by multiplication. For this construction, it is, in fact, sufficient to assume that $K$ is a near field, which is just a skew field which is distributive only from the left (see Definition \ref{NF}). The long standing conjecture in the field is that the converse is also true:
\begin{conj} \label{conj:main}
For every sharply 2-transitive permutation group $\Gamma$ there exists a near field $N$ such that $\Gamma = N^{\times} \ltimes N$. 
\end{conj}
The conjecture was confirmed by H. Zassenhaus for finite sharply $2$-transitive groups in a 1936 paper. Later in \cite{Z1} he completed the classification of all finite sharply 2-transitive groups by classifying all finite near fields. The aforementioned classification contains interesting examples as, contrary to the situation for skew fields, there are non-trivial examples of finite near fields: the so called Dickson near-fields as well as seven other sporadic cases. In the infinite case, much less has been done. In \cite{T2} Tits proved the conjecture for locally compact connected sharply $2$-transitive groups. In that case too all near fields are classified and in particular they are all of finite rank over $\R$.

An important first step is to determine the characteristic directly from the structure of the permutation group. To every sharply 2-transitive group $\Gamma$ one associates (see definition \ref{def_prim}) its permutational characteristic, $\pchar(\Gamma)$, which is equal to $\Char(N)$ whenever the conjecture above holds. Our main theorem settles the above conjecture under the assumption that $\Gamma$ is a linear group. Unfortunately we do have to exclude the case of characteristic $2$, in two different ways. 
\begin{theorem} \label{thm1}
 Let $F$ be a field and let $\Gamma\leq\GL_n(F)$ be a sharply $2$-transitive group. Assume that $\Char(F)\neq2$ and that $\pchar (\Gamma) \neq2$. Then $\Gamma\cong N^\times\ltimes N$, where $N$ is a near-field.
\end{theorem}
  The hardest part of the proof is to show that $\Gamma$ splits as a semi-direct product. The fact that in this case the group already satisfies the conjecture follows from Theorem \ref{thm_dm}.
  
The current paper is part of an attempt to implement some methods from finite group theory to the more general setting of linear groups. In this it builds on previous papers \cite{GG},\cite{GG2} in which a version of the Aschbacher-O'Nann-Scott theorem was proved for countable linear groups. In the terminology established in those papers what we show here is that sharply $2$-transitive linear groups (which are automatically primitive) must be of affine type. Eventually though, the current paper does not rely directly on \cite{GG},\cite{GG2} and in fact what we do use from these papers is proved here slightly more generally because we do not assume that all groups are countable or non-torsion.

The authors would like to thank Uri Bader, Tsachik Gelander, Yoav Segev and Barak Weiss for very helpful comments on previous versions of this work. We would like to thank Theo Grundh\"{o}fer and the referee for pointing us, in two different ways, to the fact that the commutativity of addition in a near-field follows easily and thus shortening our proof.

\section{Preliminaries} \label{sec:prelim}
\subsection{Sharply $2$-transitive groups}
\begin{rmk}The basic facts mentioned in this section can be found in \cite{K}. Our proof appears here for the convenience of the reader. 
\end{rmk}
Let $\Omega$ be a set and $\Sym (\Omega)$ be the full group of permutations, acting from the left on the set $\Omega$. If $\Gamma < \Sym(\Omega)$ is a subgroup we will denote by $\Gamma_\omega = \{\gamma \in \Gamma  \ | \ \gamma \omega = \omega\}$ the stabilizer of the point $\omega$ and by $\Gamma \omega = \{\gamma \omega \ | \ \gamma \in \Gamma\}$ the orbit. 

\begin{defi} \label{def_prim} \rm
A permutation group $\Gamma < \Sym(\Omega)$ is called {\it{sharply 2-transitive}} if it acts transitively and freely on ordered pairs of distinct points.
\end{defi}

It follows directly from the definition that such a group contains many involutions: by the transitivity on ordered pairs any two points can be flipped by a group element and the square of such a flip is trivial due to the freeness of this action. The freeness also implies that an involution is uniquely determined by any pair of points that it flips. Thus the set $\Inv(
\Gamma)$, forms a conjugacy class. Indeed if $\sigma$ flips $x,y \in \Omega$ and $\tau$ flips $z,w \in \Omega$ then any element   satisfying $(\gamma z , \gamma w)=(x,y)$ will serve as a conjugating element $\sigma^{\gamma} := \gamma^{-1} \sigma \gamma = \tau$. Consequently we distinguish between two separate cases, either every involution stabilizes a point or none does. Moreover:

\begin{lemma} \label{u.f.i}
 If $\Gamma$ is a sharply $2$-transitive group then $\Gamma_\omega$ contains at most one involution.
\end{lemma}
\begin{proof}
 Assume $\sigma,\tau \in \Gamma_\omega$ are involutions. Fix any $y \neq \omega$. Since $\Gamma$ is sharply $2$-transitive, there exists $\gamma \in \Gamma_y$ such that $\gamma (\tau y)=\sigma y$. Then we have $\sigma^\gamma y = \gamma^{-1} \sigma y=\tau y$. Thus $\tau =\sigma^{\gamma}$, as they coincide on $y$ and $\tau y$, which implies $\omega=\tau \omega =\gamma^{-1} \sigma \gamma \omega$. Thus $\sigma (\gamma \omega)=\gamma \omega$. Hence $\sigma \in \Gamma_\omega \cap \Gamma_{\gamma \omega}$. By the ``sharply" assumption, this implies $\gamma \omega=\omega$. Thus $\gamma$ fixes both $\omega,y$, which means that $\gamma=1$. \qedhere
\end{proof}
We turn to the situation where $\Gamma_\omega$ contains an involution. It will turn out that $\Inv(\Gamma)^2 \setminus \{ 1 \}$ is a conjugacy class, which, in turn, accounts for the following proposition that will be important to us.

\begin{prop} \label{prop_per}
Let $\Gamma$ be a sharply $2$-transitive group and assume that $\Gamma_\omega$ contains an involution. Then all elements in $\Inv (\Gamma)^2 \setminus \{ 1 \}$ share the same order which is either $\infty$ or a prime number $p$. 
\end{prop}

We will call an element of a permutation group {\it{fixed-point free}} if it does not fix a point. A subgroup  is called {\it{fixed-point free}} if all its non-identity elements are fixed-point free.

\begin{lemma} \label{inv_reg}
If $\Gamma$ is a sharply $2$-transitive group then all elements in $\Inv (\Gamma)^2 \setminus \{ 1 \}$ are fixed-point free.
\end{lemma}

\begin{proof}
Let $\sigma,\tau \in \Inv(\Gamma)$. Assume that there exists $\omega \in \Omega$ such that $\sigma\tau \in \Gamma_\omega$, i.e $\sigma \omega=\tau \omega$. Hence, either $\sigma \omega =\tau \omega=\omega$ which is equivalent to $\sigma, \tau \in \Gamma_\omega$, and by Lemma \ref{u.f.i}, $\sigma=\tau$. Or $\sigma \omega =\tau \omega=y \neq \omega$. Then $\sigma, \tau$ coincide on two distinct points, $\omega$ and $y$, hence are equal. \qedhere
\end{proof}

\begin{prop} \label{charn2}
 Let $\Gamma$ be a sharply $2$-transitive group and assume $\Gamma_\omega$ contains an involution. Then the given sharply 2-transitive action of $\Gamma$ is isomorphic to the action of $\Gamma$ on $\Inv (\Gamma)$ by conjugation.
\end{prop}

\begin{proof}
Let $\Gamma \action \Omega$ be a sharply 2-transitive action.
Since $\Inv(\Gamma)$ forms one conjugacy class the action $\Gamma \action \Inv (\Gamma)$ is transitive.
Now, it is sufficient to show that the stabilizers of points in both actions are equal, i.e $\forall \omega \in \Omega$ $\exists \sigma \in \Inv (\Gamma)$ such that $\Gamma_\omega = \Gamma_\sigma$, and vice versa. Take any $\omega \in \Omega$. By assumption on $\Gamma$ and by Lemma \ref{u.f.i}, the stabilizer contains a unique involution $\sigma$ which is centralized by $\Gamma_\omega$. Since $2$-transitive actions are primitive, $\Gamma_\omega$ is a maximal subgroup, so we have $\Gamma_\omega = \Gamma_\sigma$. \qedhere
\end{proof}

\begin{proof}[Proof of Proposition \ref{prop_per}]
The fact that $\Inv (\Gamma)^2 \setminus \{ 1 \}$ is one conjugacy class, follows immediately from Proposition \ref{charn2}.
$\Inv (\Gamma)^2$ is closed under powers: Take any $\sigma \tau \in \Inv (\Gamma)^2 \setminus \{ 1 \}$ and $m \in \N$. 
 If $m$ is odd then $(\sigma \tau)^m = (\sigma \tau)^{\frac{m-1}{2}} \sigma (\tau \sigma)^{\frac{m-1}{2}} \tau$ is once again a product of two involutions. 
 If $m$ is even then $(\sigma \tau)^m = \left( (\sigma \tau)^{\frac{m}{2}-1}\sigma \right) \tau \left( \sigma (\tau \sigma)^{\frac{m}{2}-1} \right) \tau$ is a product of two involutions.
 
We have just proved that all elements in $\Inv (\Gamma)^2 \setminus \{ 1 \}$ share the same order. If that order is finite, denote it by $n$. Assume, by contradiction, that $n=p \cdot q$ is composite ($p,q \neq 1$). Then, $(\sigma \tau)^q$ is another element in $\Inv (\Gamma)^2 \setminus \{ 1 \}$, whose order is $p \neq n$, which is a contradiction. \qedhere
\end{proof}

\noindent The previous proposition leads us to the definition of the following invariant:

\begin{defi} \rm Let $\Gamma$ be a sharply $2$-transitive group. The {\it{permutation characteristic}} of $\Gamma$, denoted by $\pchar(\Gamma)$, is defined as follows:
$$\pchar(\Gamma)=\left\{ \begin{array}{ c l }
	2\hspace{8pt} & \Gamma_\omega\cap\Inv(\Gamma)=\emptyset\\
	p\hspace{8pt} & \Gamma_\omega\cap\Inv(\Gamma)\neq\emptyset, \hspace{10 pt} \forall \gamma\in\Inv(\Gamma)^2: \hspace{8 pt} \ord(\gamma)=p\\
	0\hspace{8pt} & \Gamma_\omega\cap\Inv(\Gamma)\neq\emptyset, \hspace{10 pt} \forall \gamma\in\Inv(\Gamma)^2: \hspace{8 pt} \ord(\gamma)=\infty
\end{array}\right.$$
\end{defi}

\begin{rmk}
If $\Gamma_\omega\cap\Inv(\Gamma)\neq\emptyset$ then $\pchar(\Gamma)\neq2$ since if an element of $\Inv(\Gamma)^2$ is an involution, then it must fix a point, which contradicts Lemma \ref{inv_reg}.
\end{rmk}

\begin{ex} \label{ex_aff} \rm
Let $N$ be a division ring and let $\Gamma = N^{\times} \ltimes N$ with the multiplication defined by $(a,b)\cdot(a',b')=(aa',ab'+b)$. Observe the action $\Gamma \action N$ defined by $(a,b)x=ax+b$. This action is sharply $2$-transitive since for every two pairs of distinct points $x \neq y, z \neq w$ in $N$ we can find a unique element $(a,b) \in \Gamma$ such that $ax+b=z$ and $ay+b=w$ by solving a system of two linear equations. Now, let us  describe $\Inv \left( \Gamma \right)$: $(a,b) \in \Inv (\Gamma)$ $\Leftrightarrow$ $(a,b)^2=(a^2,ab+b)=(1,0)$. Hence $a\in \{ 1,-1 \}$.
 If $a=1$ then $2b=0$. Hence, if $\Char(N)=2$ then any $b \neq 0$ will give us a non-trivial involution.
   Else, $\Char(N)\neq 2$ implying $b=0$, which gives us the identity.
 If $a=-1$ then any $b$ will give us a non-trivial involution (except for $b=0$ if $\Char(N)=2$).
So $\Inv (\Gamma) = \{ (-1,b) | b \in N \}$. Now take the product of two distinct involutions: $(-1,b)\cdot (-1,c) = (1,b-c)$. Hence $((-1,b)\cdot (-1,c))^n = (1,b-c)^n=\left(1,n(b-c)\right)$ which is equal to $(1,0)$ if and only if $n(b-c)=0$. Since $b \neq c$, we must have $n=0$, which implies that  $\pchar(\Gamma) = \Char(N)$
\end{ex}
Note that in the example $\pchar(\Gamma) = \Char(N)$. However, even if $\Gamma < \GL_n(k)$ is a linear group this number need not coincide with the characteristic of the field. Indeed if $\Gamma$ is a finite sharply two transitive group it can be realized as a linear group over any field. 

The main Conjecture \ref{conj:main} claims that, upon allowing $N$ to be a near field rather than a division ring, the above construction yields all sharply two transitive groups. 

\subsection{Near-Fields}
Observing Example \ref{ex_aff} more closely, one sees that for the action $\Gamma = N^{\times} \ltimes N \action N$ remains sharply 2-transitive, even if $N$ is not required to satisfy right distributivity. In other words, $N$ only has to be a near-field.
\begin{defi} \label{NF} \rm
A $5$-tuple $\langle N,+,\cdot,0,1 \rangle$ is called a {\it{near-field}} if it satisfies the following $4$ axioms:
\begin{description}
\item[NF1:] $\langle N,+,0 \rangle$ is an abelian group.
\item[NF2:] $\langle N \setminus \{ 0 \},\cdot,1 \rangle$ is a group.
\item[NF3:] For all $a,b,c \in N$ we have $a \cdot (b+c) = a \cdot b + a \cdot c$.
\item[NF4:] $a \cdot 0=0 \cdot a=0$ For all $a \in N$.
\end{description}
\end{defi}

\begin{rmk}
 This structure is sometimes called a left near-field. One can define a right near-field by replacing {\textbf{NF3}} with right distributivity.
\end{rmk}
\begin{rmk}\label{nab}
 The assumption of commutativity in {\textbf{NF1}} is non-essential, since it follows from the other axioms. This was first proved for finite near-fields by Zassenhaus in \cite{Z2}, and later for infinite near-fields by B.H. Neumann in \cite{N}.
\end{rmk}

\begin{theorem}\rm{\cite[Theorem 7.6C]{DM}} \label{thm_dm}
 Let $| \Omega| \geq 2$ and let $G \leq \Sym(\Omega)$ be a sharply $2$-transitive group which possesses a fixed-point free normal abelian subgroup $K$. Then there exists a near-field $N$ such that $G$ is permutation isomorphic to $N^\times \ltimes N$.
\end{theorem}
\noindent For the convenience of the reader, we shall add a sketch of the proof:
\begin{proof}
Set $N=\Omega$. Fix two arbitrary elements in $N$ and denote them by $0,1$. Since $K$ is fixed-point free, the map $k \mapsto k(0)$ defines a bijection between $K$ and $N$. Define addition on $N$ by $k(0)+m(0)=km(0)$, which turns the map above to an isomorphism from $K$ to $\langle N,+ \rangle$. Clearly, {\textbf{NF1}} holds. 
Since $G_0$ is fixed-point free on $\Omega \setminus \{0\}$, now using the map $g \mapsto g(1)$, we can define multiplication on $N^\times=N\setminus\{0\}$ by $g(1) \cdot h(1)=gh(1)$. With this definition, $G_0$ and $\langle N^\times,\cdot \rangle$ are isomorphic, which implies {\textbf{NF2}}. 
Define $a \cdot 0=0 \cdot a=0$ For all $a \in N$. This guarantees {\textbf{NF4}}. \\
$K$ is normal, so we have $k^g \in K$ for all $k \in K$ and $g \in G_0$. Also, $g(1) \cdot k(0)=k^{g^{-1}}(0)$:
\begin{itemize}
 \item If $k=1_G$ then $k^g=1_G$ which implies $g(1) \cdot k(0)=g(1) \cdot 0=0=k^{g^{-1}}(0)$.
 \item If $k \neq 1_G$ then there exists a unique $h \in G_0$ such that $k(0)=h(1)$. Then $g(1) \cdot k(0)=g(1) \cdot h(1)=gh(1)=gk(0)=k^{g^{-1}}(0)$, since $g \in G_0$.
\end{itemize}
Now take any $a,b,c \in N$. 
\begin{itemize}
 \item[$\circ$] If $a=0$ then $0 \cdot (b+c)=0=0+0=0 \cdot b + 0\cdot c$.
 \item[$\circ$] If $a \neq 0$ then $a=g(1),b=k(0),c=m(0)$ where $g \in G_0$ and $k,m \in K$. We have $$a \cdot (b+c)=g(1) \cdot (k(0)+m(0))=g(1) \cdot km(0)=(km)^{g^{-1}}(0)=k^{g^{-1}}m^{g^{-1}}(0)$$ $$=k^{g^{-1}}(0)+m^{g^{-1}}(0)=g(1) \cdot k(0)+ g(1) \cdot m(0)=a \cdot b + a \cdot c$$
\end{itemize}
This proves {\textbf{NF3}}, and completes the proof that $N$ is a near-field. Moreover:
\begin{itemize}
 \item[$\odot$] $K$ acts on $N$ by: $k(a)=km(0)=k(0)+m(0)=a+k(0)$.
 \item[$\odot$] $G_0$ acts on $N$ by: $g(a)=gh(1)=g(1)\cdot h(1)=g(1)\cdot a$. (If $a=0$ then $g(0)=0=g(1)\cdot 0$).
\end{itemize}
Since $K$ is a fixed-point free normal subgroup, every $\gamma \in G$ can be written as $\gamma=gk$ for some $g \in G_0$ and $k \in K$. Hence we have $$\gamma(a)=gk(a)=g(a+k(0))=g(1)\cdot(a+k(0))=g(1)\cdot a+g(1)\cdot k(0)$$ So $G$ acts on $N$ as $x \mapsto a\cdot x +b$ for $a \in N \setminus \{0\}$ and $b \in N$. Since $G$ is $2$-transitive, it contains all the permutations of this form.
\end{proof}

\subsection{Algebraic Groups}
When dealing with algebraic groups, the corresponding notion of transitivity is a bit different:
\begin{defi}\label{def:gtd}\rm Let $\rho:G\action X$ be an algebraic group acting algebraically on an algebraic variety $X$. $\rho$ is called {\it{generically $n$-transitive}} if the action $\rho^n$ of $G$ on $X^n$ admits an open dense orbit. If the action itself does not admit an open dense orbit we will say that it is generically $0$-transitive. The {\it{generic transitivity degree}} of the action $\rho$ is defined by  $\gtd(\rho)=\sup\LR{n:\rho\mbox{ is generically } n\mbox{-transitive}}.$ The {\it{generic transitivity degree}} of the group $G$ is defined by $\gtd(G)=\sup\LR{\gtd(\rho):\rho\mbox{ is an algebraic action of } G}.$
\end{defi}

  In \cite{P}, V.L.Popov gives a complete classification of algebraic groups by their generic transitivity degree. We will cite a few results from this paper, which are important for us, but are far from being the main results there.
\begin{lemma}\rm{\cite[Lemma 2]{P}} \label{p2}
 Let $\ga_i$ be an action of a connected algebraic group $G$ on an irreducible algebraic variety $X_i$, $i=1,2$. Assume that there exists a $G$-equivariant dominant rational map $\varphi:X_1 \dashrightarrow X_2$. Then $\gtd(\ga_1)\leq\gtd(\ga_2)$. Moreover, if $\dim(X_1)=\dim(X_2)$ then $\gtd(\ga_1)=\gtd(\ga_2)$.
\end{lemma}
\begin{cor}\label{p2c}
 Let $G$ be a connected algebraic group $G$ and $H<M<G$ be two closed subgroups. Then $\gtd(G:G/H)\leq\gtd(G:G/M)$.
\end{cor}

It turns out that an action of a reductive group with reductive stabilizers cannot be too transitive. This result, which will be central to our argument, was first observed by Popov \cite[Lemma]{P}, in characteristic 0, as a corollary to the main theorem of Luna in \cite{L}. In general characteristic the result was proved by Brundan in \cite[Theorem 1]{Br1} (see also \cite[Corollary 6.8]{Se},\cite{Br2},\cite{Br3}). 
\begin{theorem} \label{lp} 
If $H < G$ are algebraic groups, both reductive. Then $\gtd(G,G/H) = 1$. 
\end{theorem}
\begin{proof}
In fact we will prove a little more by showing that, under the same assumptions, there are no open orbits in $G/H \times G/H$. Since the action is transitive, we have $\gtd(G,G/H) \ge 1$.  Assume, by way of contradiction, that the action of $H$ on $G/H$ admits an open orbit. 

Since $[H:H^0] < \infty$ every $H$ orbit, and in particular the open orbit, is a finite union of $H^0$ orbits. Hence at least one of the $H^0$ orbits, say $\ol{O}=H^0gH$, is open. Since $H^0$ is connected so are its orbits, so that $\ol{O}$ is contained the connected component $\ol{X} := G^0 g H$ of $G/H$. Consider the natural map $\phi: G/H^0 \arrow G/H$ and restrict it to the connected component, $X = G^0 g H^0$. Since $[H:H^0] < \infty$  the map $\phi: X \arrow \ol{X}$ is a covering map and hence $O := H^0 g H^0 = \phi^{-1}(\ol{O}) \cap X$ is still open in $X$ and since $X$ is connected this set is also dense. Thus $H^0 g H^0$ is an open dense double coset in $G^0$ which contradicts \cite[Theorem 1]{Br1}. 
\end{proof}

\begin{rmk}
 In case of $\Char(F)=0$ the proof can be ended alternatively. Since both $G^0,H^0$ are  reductive and connected it follows from \cite[Main theorem]{L} that a generic point in $G^0/H^0$ has a closed $H^0$ orbit. Since we saw that there is an open dense orbit, the action $H^0 \curvearrowright G^0/H^0$ must be transitive which is absurd.  
\end{rmk}

\section{Proof of the theorem}
Throughout this section, we let $F$ be a field with $\Char(F)\neq2$ and $\Gamma<\GL_n(F)$ is a sharply $2$-transitive group, with $\pchar(\Gamma)\neq2$, acting sharply $2$-transitively on a set $\Omega$. Since the theorem is known for finite sharply 2-transitive groups \cite{Z1} we assume that $\Gamma$ is infinite. 

\begin{theorem}\label{banal} \rm
Let $\Gamma<\GL_n(F)$ be a sharply $2$-transitive group. Assume that $\Char(F)\neq2$ and that $\pchar(\Gamma)\neq2$. Then there exist a non-trivial abelian normal subgroup $N\unlhd \Gamma$.
\end{theorem}

\begin{prop}\label{7.1}
Let $\Gamma$ be as in the assumptions of Theorem \ref{banal}. If the conclusion of Theorem \ref{banal} fails, then there exists an algebraically closed field $k$ and a faithful representation $\rho:\Gamma\to\GL_n(k)$ such that $\GG=\ol{\rho(\Gamma)}^Z$ is reductive.
\end{prop}
\begin{proof}
Let $\rho_0:\Gamma\to\GL_n(k)$ be a faithful representation over some algebraically closed field $k$. Let $\GG_0=\overline{\rho_0(\Gamma)}^Z$ be the Zariski closure and let $\GG_u$ be the unipotent radical of $\GG$. Let $N=\rho_0(\Gamma)\cap\GG_u$. Since $N$ is nilpotent, if it is non-trivial then the penultimate element of the lower central series of $N$ is a normal abelian subgroup of $\rho(\Gamma)$, which contradicts the assumption. Hence we can divide by $\GG_u$ and obtain a new faithful representation $\rho$, for which $\GG=\overline{\rho(\Gamma)}^Z$ is reductive.
\end{proof}

The following Proposition is simple but central because it allows us to pass from the given sharply $2$-transitive action to a generically $2$-transitive action of the Zariski closure. It is here where our assumption on the permutational characteristic is used to show that the latter action is not trivial.  The assumption on the characteristic of the field is also used to show that the stabilizer of a point in the latter action is contained in a proper reductive subgroup. 
\begin{prop}\label{orig}
Let $\Gamma$ and $F$ be as in the assumptions of Theorem \ref{banal}. Let $k$ be the algebraic closure of $F$, $G, H < \GL_n(k)$ the Zariski closures of $\Gamma$ and $\Gamma_\omega$ respectively, and $\sigma \in \Gamma_\omega$ the unique involution. Then $\sigma$ is semi-simple, $H \leq C_G(\sigma)$ and $\gtd(G:G/H)\geq2$.
\end{prop}
\begin{proof}
By Lemma \ref{u.f.i}, $\sigma$ is the unique involution in $\Gamma_\omega$, it is centralized by all elements of $\Gamma_\omega$. Commuting with a given element is a Zariski closed condition, hence it passes to $\ol{(\Gamma_\omega)}^Z=H$, so $H\leq C_G(\sigma)$. Since $\Char(k)=\Char(F)\neq2$, $\sigma$ is semi-simple as its minimal polynomial divides $x^2-1$ and is thus a product of linear factors with multiplicity 1.

Let $\gamma \in \Gamma \setminus \Gamma_\omega$ be any element not in $\Gamma_\omega$. By definition of $2$-transitivity we have $\Gamma / \Gamma_\omega = \Gamma_\omega \sqcup \Gamma_\omega \gamma \Gamma_\omega$. Now consider the set $H \sqcup H \gamma H < G/H$. This is a Zariski dense set since it contains the dense set $\Gamma H  = H \sqcup \Gamma_\omega \gamma H$ (which is dense since $\bar{\Gamma}^Z=G$); moreover, being an orbit of an algebraic action, it is locally closed (i.e. open in its closure) by \cite[Proposition 8.3]{Hu}. Thus $H \sqcup H\gamma H$ is open in $G/H$, and since points are closed so is $H \gamma H$. Consider the projection map $\pi: G/H \times G/H \arrow G/H$ which is equivariant with respect to the natural left $G$-action. We have just seen that the intersection of the $G$-orbit $G(H,\gamma H)$ with the fiber $\pi^{-1}(H)$ is open and dense in the fiber. Since the action of $G$ on $G/H$ is transitive this is actually true for every fiber. Hence, the $G$-orbit $G(H,\gamma H) \subset G/H \times G/H$ is dense. Since all orbits are automatically open in their closures we are finished. 
 \qedhere
\end{proof}

\begin{proof}[Proof of Theorem \ref{banal}]
Assume by contradiction that there is no non-trivial abelian normal subgroup of $\Gamma$. By Proposition 
\ref{7.1} we may assume that $\Gamma < \GL_n(k)$ where $k$ is an algebraically closed field and that the Zariski closure $G =\ol{\Gamma}^Z$ is reductive. Now let $\Gamma_\omega$ be the stabilizer of a point in the sharply $2$-transitive action, $H=\ol{\Gamma_\omega}^Z$ its Zariski closure and $\sigma \in \Gamma_\omega$ the unique involution in $\Gamma_\omega$. Such an involution exists as we have assumed that $\pchar(\Gamma) \ne 2$. By Proposition \ref{orig} the action on the level of the Zariski closures $G \curvearrowright G/H$ is generically 2-transitive. Setting $C = C_G(\sigma)$ we conclude using Corollary \ref{p2c} that $$\gtd(G: G/C) \ge \gtd(G: G/H) \ge 2.$$ By Proposition \ref{orig}, using the characteristic assumption on the field this time, $\sigma$ is semisimple and hence $C = C_G(\sigma)$ is reductive by \cite[Proposition 13.19]{B}. This stands in contradiction to Theorem 
\ref{lp} and concludes the proof. 
\end{proof}

\noindent We are now ready to prove our main theorem. 
\begin{proof}[Proof of Theorem \ref{thm1}]
By Theorem \ref{banal}, $\Gamma$ admits a non-trivial abelian normal subgroup $N\unlhd \Gamma$. Take any $\omega\in\Omega$. Since $N$ is abelian, $[N_\omega,N]=\LR{e}$, but $N$ is transitive - hence $N_\omega=\LR{e}$. So $N$ is a normal, fixed-point free abelian subgroup. Theorem \ref{thm_dm} applies and concludes the proof.
\end{proof}

\noindent\textbf{Acknowledgments.} Both authors were partially supported by ISF grant 441/11.

\noindent\textsc{Yair Glasner.} Department of Mathematics. Ben-Gurion
University of the Negev. P.O.B. 653, Be'er Sheva 84105, Israel.
\texttt{yairgl\@@math.bgu.ac.il}\bigskip

\noindent\textsc{Dennis D. Gulko.} Department of Mathematics. Ben-Gurion
University of the Negev. P.O.B. 653, Be'er Sheva 84105, Israel.
\texttt{gulkod\@@math.bgu.ac.il}\bigskip

\end{document}